\documentclass[12pt]{amsart}
\usepackage[english]{babel}
\usepackage[utf8]{inputenc}
\usepackage{amsmath}
\usepackage{amsthm}
\usepackage{graphicx}
\usepackage{amssymb}
\usepackage{enumerate}
\usepackage{enumitem}
\usepackage{color}
\usepackage{accents}

\newtheorem{theorem}{Theorem}

\newtheorem{lemma}[theorem]{Lemma}

\newtheorem{definition}[theorem]{Definition}

\newcommand{\B}{\mathbb{B}}
\newcommand{\D}{\mathbb{D}}
\newcommand{\N}{\mathbb{N}}

\newcommand{\C}{\mathbb{C}}

\newcommand{\eps}{\varepsilon}
\newcommand{\supp}{\mathrm{supp\;}}

\newenvironment{proof*}{\vskip 2mm\noindent {}}{\hfill $\Box$ \vskip 2mm}

\begin{document}

\title[Cyclicity in the polydisc and ball]{Cyclicity of non vanishing functions in the polydisc and in the ball}

\author{Eric Amar and
Pascal J. Thomas}

\address{E. Amar: Universit\'e de Bordeaux\\
351 Cours de la Lib\'eration, Talence\\ France}
\email{Eric.Amar@math.u-bordeaux.fr}

\address{P.J. Thomas: Universit\'e de Toulouse\\ UPS, INSA, UT1, UTM \\
Institut de Math\'e\-ma\-tiques de Toulouse\\
F-31062 Toulouse, France} 
\email{pascal.thomas@math.univ-toulouse.fr}


\subjclass[2000]{}

\thanks{}

%\date{\today}

\begin{abstract}
We use a special version of the Corona Theorem in several variables,
valid when all but one of the data functions are smooth, to generalize to
the polydisc and to the ball results obtained by El Fallah, Kellay and Seip about
cyclicity of non vanishing bounded holomorphic functions in large enough
Banach spaces of analytic functions determined either by weighted sums of powers of Taylor coefficients
or by radially weighted integrals of powers of the modulus of the function. 
\end{abstract}

\maketitle

\section{Introduction}
The Hardy space can be seen as a space of square integrable functions on the circle with vanishing Fourier coefficients for the negative integers, a space of holomorphic functions on the unit disk, or the
space of complex valued series with square summable moduli, and the interaction between those viewpoints
has generated a long and rich history of works in harmonic analysis, complex function theory and operator
theory.

The present work aims at generalizing one particular aspect of this to several complex variables: the
study of cyclicity of some bounded holomorphic functions under the shift operator in large enough Banach spaces containing the Hardy space. 

\subsection{Definitions.}
\begin{definition}
\label{series}
Let $\omega : \N^d \longrightarrow (0,\infty)$, where $d\in \N^*$, and $p\ge 1$. We define
the Banach space of power series in several variables
$$
X_{\omega, p} := \left\{ f(z):= \sum_{I\in  \N^d} a_I z^I : \|f\|_{X_{\omega, p}}^p :=
\sum_{I\in  \N^d} \left( \frac{|a_I|}{\omega(I)}\right)^p <\infty
\right\},
$$
with the usual multiindex notation, $z=(z_1, \dots, z_d)\in \C^d$, $I=(i_1, \dots, i_d) \in \N^d$,
$z^I:= z_1^{i_1} \cdots z_d^{i_d}$. 
\end{definition}
We also write $|I|:= i_1 + \cdots + i_d$, $I!:= i_1! \cdots i_d!$. 
We say that $\omega$ is \emph{nondecreasing} if for any $I, J$, $\omega (I+J) \ge \omega(J)$. 

Recall that domains of convergence of power series are logarithmically convex complete
Reinhardt domains (for a definition and those terms and proofs, see e.g. \cite{Kra}, \cite{Horm}).
In what follows, we shall restrict our attention to the cases of the polydisc 
$\D^d:= \{ z \in \C^d: \max_{1\le j \le d} |z_j| <1 \}$ and the unit ball
$\B^d:= \{ z \in \C^d: \sum_{1\le j \le d} |z_j|^2 <1 \}$. The letter $\Omega$
will stand for either one of those two domains, except in the more general
Theorem \ref{corona}.

If $\omega(I)=1$ for any $I$, then we obtain the Hardy space $H^2(\D^d)$, which can also
be described as the set of functions in the Nevanlinna class of the polydisc
with 
boundary values (radial limits a.e.) on the torus $(\partial \D)^d$ which are in $L^2((\partial \D)^d)$,
and
$$
\|f\|_{H^2(\D^d)}^2 = \sum_{I\in \N^d } |a_I|^2 = \frac1{(2\pi)^d}\int_{(\partial \D)^d} |f|^2 d\theta_1 \dots d\theta_d.
$$
The standard references for Hardy spaces on polydiscs is \cite{RudinPolydiscs69}.

There is a Hardy space for $\B^d$, which is most easily described as 
as the set of functions in the Nevanlinna class of the ball
with 
boundary values (radial limits a.e.) on the sphere $\partial \B^d$ which are in $L^2(\partial \B^d)$,
and
$$
\|f\|_{H^2(\B^d)}^2 =  \int_{\partial \B^d} |f|^2 d\sigma,
$$
where $\sigma$ is the $(2d-1)$-real dimensional Lebesgue measure normalized so that $\sigma (\partial \B^d)=1$. 
The standard reference for function theory on the unit ball is \cite{RudinBall}. 
Lemma \ref{hardyweight} gives a description of $H^2(\B^d)$ in terms of the coefficients in the
Taylor expansion.
%Note that $\|z^J\|^2_{H^2(\B^d)}= \frac{(d-1)! J!}{(|J|+d-1)!}$ \cite[p. 12]{RudinBall}.

\begin{definition}
\label{stdweight}
We set 
$\omega_2^{\D^d}(J):= 1$, and
\[
\omega_2^{\B^d}(J):=\frac1{\|z^J\|_{H^2(\B^d)} }= \left( \frac{(|J|+d-1)!}{(d-1)! J!} \right)^{1/2}.
\]
\end{definition}
We sometimes use the notation $\omega_2(J)$ (without superscript) when either of those quantities is meant.

We still have to understand in what sense a power series $f$ can be understood as a function of $z \in \Omega$.
We will want to consider weights which satisfy the following relative monotonicity condition :
 there exists a constant $C_m\ge 1$ such that, for any $I, J \in \N^d$,
\begin{equation}
\label{relmono}
C_m \omega (I+J) \ge \omega(J) \omega^\Omega_2 (I).
\end{equation}
%In the case of the polydisk, 
One can check that $\omega=\omega_2^{\B^d}$ itself verifies condition \eqref{relmono}. 

When $\Omega =\D^d$, then  $\omega^\Omega_2 (I)=1$ and if
$C_m=1$, we recover the usual monotonicity.
We observe that for the polydisc, we can reduce ourselves to the case $C_m=1$.

\begin{lemma}
\label{relabs}
Let $\Omega=\D^d$. If $\omega$ satisfies Condition \eqref{relmono}, then $X_{\omega, p}$
admits an equivalent norm given by the \emph{nondecreasing} weight
\[
\hat \omega (I) := \inf_{J \in \N^d} \omega(I+J).
\]
\end{lemma}
\begin{proof}
Since $0\in \N^d$ and we have \eqref{relmono}, $1 \ge \frac{\hat \omega (I)}{\omega (I)} \ge C_m^{-1}$, and 
\[
\hat \omega (I+K) = \inf_{J \in \N^d} \omega(I+K+J) = \inf_{J \in K+\N^d} \omega(I+J)
\ge \inf_{J \in \N^d} \omega(I+J) = \hat \omega (I).
\]
\end{proof}
Since the new norm is equivalent to the original one, the problem is unchanged
and there is no loss of generality in assuming that $\omega$ has been modified and made 
nondecreasing, and we shall do so henceforth. 

\begin{lemma}
\label{lemzero}
Let $\Omega = \D^d$ or $=\B^d$.

If $\omega$ verifies the relative monotonicity condition \eqref{relmono} and 
\begin{equation}
\label{weightcond}
\log \omega_2^\Omega(I) \le \log \omega(I) \le \log \omega_2^\Omega(I) + o(|I|), 
\end{equation}
for any $f \in X_{\omega, p}$, the series defining $f$ converges on $\Omega$,
%and the point evaluations $f\mapsto f(z)$ are continuous on $X_{\omega, p}$ for each $z\in \Omega$.
and the map 
%\newline
$X_{\omega, p} \ni f \mapsto f(z)$ is  continuous with respect to the norm
$ \|\cdot\|_{X_{\omega, p}}$. 
In particular, $X_{\omega, p}$ can be seen as a subset of the
space $\mathcal H (\Omega)$ of holomorphic functions on $\Omega$. 

Furthermore,
 there is no larger domain on which every $f \in X_{\omega, p}$ has to be holomorphic.
\end{lemma}
This lemma will be proved in Section \ref{ConvDom}.

In the ball case, consider $\lambda $  a probability measure on $[0,1)$.
%such that $\int_0^1 d\lambda(r)=1$.
\begin{definition}
\label{bergmanball}
The \emph{radially weighted Bergman space} associated to $\lambda$ is
\begin{multline*}
{\mathcal{B}} ={\mathcal{B}}^{p}(\lambda )={\mathcal{B}}^{p}(\lambda )(\B^d)
\\
:=\left\{
 f\in {\mathcal{H}}({\mathbb{B}}^{d})
 :{\left\Vert{f}\right\Vert}_{p}^{p}
 :=\int_0^1 \int_{\partial \B^d}
 \left\vert f(r\zeta) \right\vert^p d\sigma(\zeta) d\lambda (r)
 <\infty \right\} .
\end{multline*}
\end{definition}
Typical examples are provided by $d\lambda(r)= c_\alpha (1-r^2)^\alpha r^{2d-1}dr$,
where $\alpha >-1$ and $c_\alpha$ is an appropriate normalizing constant; they
correspond to a weight $c_\alpha(1-\sum_{1\le j \le d} |z_j|^2)^\alpha$, $z\in \B^d$.

Let $\lambda $ be a probability measure on $[0,1)^d$, the elements of which
are denoted $r:= (r_1,\dots,r_d)$.  
%The unit circle is denoted by $\mathbb T$,
%so that $\mathbb T^d$ is the distinguished boundary of $\D^d$, 
The torus $(\partial \D)^d$ is endowed with its
normalized Haar measure denoted by $d\theta$.

\begin{definition}
\label{bergman}
The \emph{weighted Bergman space} associated to $\lambda$ is
\begin{multline*}
{\mathcal{B}}={\mathcal{B}}^{p}(\lambda )={\mathcal{B}}^{p}(\lambda )(\D^d)
\\
:=\left\{
 f\in {\mathcal{H}}({\mathbb{D}}^{d})
 :{\left\Vert{f}\right\Vert}_{p}^{p}
 :=\int_{[0,1)^d} \int_{\mathbb T^{d}}
 \left\vert f(r_1 e^{i\theta_1},\dots,r_d e^{i\theta_d}) \right\vert^p d\theta d\lambda (r)
 <\infty \right\} .
\end{multline*}
\end{definition}

Let $H^\infty(\Omega)$ stand for the set of bounded holomorphic functions on  $\Omega$.
In each case, the conditions on $\lambda$ ensure that $H^\infty(\Omega) \subset {\mathcal{B}}^{p}(\lambda )$.

Note that the norms of the monomials are given by moments of the measure $\lambda$.
In the case where $\Omega=\D^d$,
$$
\|z^I\|_p^p =  \int_{[0,1)^d} r^{pI} d\lambda(r ), 
$$
so that $\log (\|z^I\|_p^{-1})$ is a concave function of $I$. 

When $p=2$, % (and $\Omega=\D^d$),
${\mathcal{B}}^{2}(\lambda )$ is a Hilbert space and
the monomials $z^I$ form an orthogonal system. Notice that
in $X_{\omega,2}$, $\|z^I\|_{\omega,2} = \omega(I)^{-1}$, so that 
$$
{\mathcal{B}}^{2}(\lambda )(\D^d) = X_{\omega,2} \mbox{ with } \omega(I) =  
\left(  \int_{[0,1)^d} r^{2I} d\lambda(r ) \right)^{-1/2}.
$$

In the case where $\Omega=\B^d$, 
$$
\|z^I\|_p^p = \left( \int_0^1 r^{p|I|} d\lambda(r ) \right) \left(\int_{\partial \B^d} |\zeta^I|^p d\sigma(\zeta)\right). 
$$
When $p=2$, since the surface measure $d\sigma$ on $\partial \B^d$ desintegrates as an integral of 
Haar measures on tori,  the monomials $(z^J)$ again form an orthogonal system, and in this case
$$
\|z^I\|_2^2=
\left( \int_0^1 r^{2|I|} d\lambda(r ) \right) \omega_2^{\B^d}(J)^{-2},
$$
so that 
$$
{\mathcal{B}}^{2}(\lambda )(\B^d) = X_{\omega,2} \mbox{ with } \omega(I) =  
\left( \int_0^1 r^{2|I|} d\lambda(r ) \right)^{-1/2} \omega_2^{\B^d}(J).
$$
In general, whenever we consider a space $X$, we define the corresponding weight by
$\omega (J):=1/{\left\Vert{z^{J}}\right\Vert}_X$.

\subsection{Main results.}

Let $X$ be a Banach space as above, defined by power series or as a weighted Bergman space.
\begin{definition}
We say that a function $f \in X$ is \emph{cyclic} if for any $g\in X$, there exists
a sequence of holomorphic polynomials $(P_n)$ such that $\lim_{n\to\infty}\|g-P_nf\|_X=0$.
\end{definition}
Note that using the word ``cyclic" is a slight abuse of language, since for $d\ge 2$
we are not iterating a single operator, but taking compositions of the multiplication
operators by each of the coordinate functions $z_1, \dots, z_d$. 
It is, however, a straightforward generalization
of the usual notion of cyclicity under the shift operator $f(z)\mapsto zf(z)$. 

By Lemma \ref{lemzero} in the case of
power series spaces,
or by the mean value inequality in the case of Bergman spaces, 
the point evaluations are continuous, therefore any cyclic $f$ must verify that $f(z)\neq 0$
for any $z \in \Omega$. 

\begin{definition}
\label{omtilde}
\begin{enumerate}
\item
When $\Omega=\D^d$, for any $k\in \N$, let 
$$
\frac1{\tilde \omega (k)} := \sum_{j=1}^d \| z_j^{k}\|_X = \sum_{j=1}^d \| z^{ke_j}\|_X,
$$
where $(e_j)$ stands for the elementary multiindices of $\N^d$: 
$e_1 = (1, 0, \dots, 0)$, $e_2= (0,1,0,\dots,0)$, etc, so $ke_j= (0,\dots,0,k,0,\dots,0)$, with $k$ in the $j$-th place.
\item
When $\Omega=\B^d$, for any $k\in \N$, let 
$$
\frac1{\tilde \omega (k)} :=\sum_{J,\ \left\vert{J}\right\vert
 =k}{\frac{p_{J}}{\omega (J)}},
$$
where $p_J$ is the multinomial coefficient, $p_J:= \frac{|J|!}{J!}$.
\end{enumerate}
\end{definition}

When $X = X_{\omega, p}$ and $\Omega=\D^d$, notice that 
\[
d^{-1} \min_{1\le j \le d} \omega(ke_j) \le \tilde \omega (k)
\le \min_{1\le j \le d} \omega(ke_j).
\]
 Note that if $\omega$ 
satisfies \eqref{weightcond} then
$\log \tilde \omega (k) = o(k)$, but the converse does not hold when $d>1$.

Here are two interesting special cases of our results.

\begin{theorem}
\label{mainthm}
Let $\Omega:= \D^d$ or $\B^d$.

Suppose that $\lim_{k\to\infty} \tilde \omega (k) = \infty$, and  $\omega$ satisfies \eqref{relmono} and \eqref{weightcond},
and that
\begin{equation}
\label{div2}
\sum_{k\ge 1} \left( \frac{\log \tilde \omega (k)}{k} \right)^2 = \infty.
\end{equation}
Let $U \in H^\infty(\Omega)$, verifying $U(z)\neq 0$
for any $z \in \Omega$. 
\begin{itemize}
\item (i)
If $d\ge 1$, $p\ge 1$ and $X= {\mathcal{B}}^{p}(\lambda )$, then $U$ is cyclic in $X$.
\item (ii)
If $\Omega=\D^d$ and $X=X_{\omega,2}$, %and $\omega$ is (relatively) nondecreasing, 
then $U$ is cyclic in $X$.
\end{itemize}
\end{theorem} 

When we demand a growth condition of a slightly stronger nature on $\tilde \omega$,
we can expand the range of spaces where the result applies.

\begin{theorem}
\label{exproot}
Let $X=X_{\omega,p}$, with $p\ge 2$, or $X= {\mathcal{B}}^{p}(\lambda )$. %There exists a constant $B>0$, depending on $d$ and $p$, such that if 
If $\lim_{k\to\infty} \tilde \omega (k) = \infty$, and  $\omega$ satisfies \eqref{relmono}, \eqref{weightcond}
and
%and there exists an increasing sequence $n_k \to\infty$ such that
\begin{equation}
\label{unifB}
\limsup_k \frac{\log \tilde \omega(k)}{\sqrt k} = \infty,
\end{equation}
then any zero-free $U\in H^\infty (\Omega)$ is cyclic in $X$. 
\end{theorem} 

\subsection{Previous results.}

Many results have been proved for the case $d=1$, and even more for $p=2$. 
In one dimension, $\Omega=\D$ and $\omega = \tilde \omega$ of course. When furthermore $p=2$,
$X_{\omega,2}$ has a norm equivalent to that of a Bergman space if and only if 
$\log \omega(n)$ is a concave function of $n$ \cite[Theorem A.2 and Proposition 4.1]{BorHed}.

In his seminal
monograph \cite{nik}, N. K. Nikolski proved that if $\omega$ is nondecreasing, 
$\lim_{k\to\infty}  \omega (k) = \infty$, $\log \omega (k) = o(k)$, 
$\log \omega(n)$ is a concave function of $n$
and
\begin{equation}
\label{div32}
\sum_{k\ge 1}  \frac{\log \tilde \omega (k)}{k^{3/2}}  = \infty,
\end{equation}
then any zero-free $f\in H^\infty (\D)$ is cyclic in $X_{\omega,2}$. 

Our main inspiration comes from \cite{Seip11}, where O. El Fallah, K. Kellay and K. Seip show, still for 
$d=1$ and $p=2$, that \eqref{div2}, with no condition of concavity, is enough 
to imply cyclicity of any nonvanishing bounded function. Even though \eqref{div2} 
is a stronger condition than  \eqref{div32}, the concavity condition means that there
exist weights to which the new result applies while Nikolski's cannot \cite[Remark 2]{Seip11}. 

The novelty in the present work is of course that we have several variables,
and exponents $p\neq 2$. We also notice that it is not necessary to make use of the inner-outer
factorization: the much easier Harnack inequality suffices. 

\subsection{A Corona-like Theorem.}

As in \cite{Seip11},
our main tool is a version of the Corona Theorem. In full generality, this is still a vexingly open
question in several variables, be it in the ball or the polydisc. However, 
following an earlier result
of Cegrell \cite{Ceg}, a simpler proof \cite{Amar} gives a
Corona-type result in the special case where most of the given generating functions
are smooth. That result   is enough to yield the required estimates in this
instance. For $\Omega$ a bounded domain in $\C^d$, let 
$A^1(\Omega):= \mathcal H(\Omega) \cap \mathcal C^1(\overline \Omega)$. 
%As usual, $H^\infty (\Omega)$ stands for the space of bounded holomorphic
%functions on $\Omega$, and $\|g\|_\infty:= \sup_{z\in\Omega} |g(z)|$. 

\begin{theorem}
\label{corona}
Let $\Omega$ be a bounded pseudoconvex domain in $\C^d$,
such that the equation $\bar \partial u=\omega$, $1\le q \le n$,
admits a solution  $u=S_{q}\omega \in L_{(0,q-1)}^{\infty}(\Omega )$
when $\bar \partial \omega =0$, $\omega \in L_{(0,q)}^{\infty }(\Omega )$, 
%with $u\in L^{\infty }(\Omega )$ if $\omega \in L^{\infty }(\Omega ),$ 
 with the bounds :
 $$ 
 {\left\Vert{u}\right\Vert}_{\infty}\leq E_{q}{\left\Vert{\omega }\right\Vert}_{\infty }.
 $$

There exists a constant $C=C(d,\Omega)$ such that if $N\in \N$, $N\ge 2$,
and if $f_j \in A^1(\Omega)$, $1\le j\le N-1$, $f_N \in H^\infty (\Omega)$,
verify
$$
\sup_{z\in \Omega} \max_{1\le j \le N} |f_j(z)| \le 1,
\quad
\inf_{z \in \Omega}\sum_{j=1}^N |f_j(z)| \ge \delta >0, 
$$
then there exist $g_1, \dots, g_N \in H^\infty (\Omega)$ such that $\sum_{j=1}^N f_j(z)g_j(z)=1$
and for $1\le j\le d$,
$$
\max_{1\le j\le N}\|g_j\|_\infty \le 
C(d,\Omega) N^{4d+2} \frac{\max_{1\le j\le N-1} \| \nabla f_j\|_\infty^{d} }{\delta^{2d+1}}.
$$
\end{theorem}

Note that the polydisc and the ball verify the hypotheses of the theorem. 

\subsection{Structure of the paper.}
First we clarify the easy relationship between weights and  domains of convergence in Section \ref{ConvDom}.
Then we gather some preliminary results and a first reduction of the problem in Section \ref{auxil}.
Theorem \ref{corona} is proved in Section \ref{coronita}, and used in the proofs of 
the two main theorems.
The relatively easy proof of Theorem \ref{exproot} is given in Section \ref{proof1}. 
Theorem \ref{mainthm} will follow from a more general and more technical result, Theorem \ref{metathm},
which is stated and proved in Section \ref{proof2}. 

\section{Domains of convergence}
\label{ConvDom}

\begin{lemma}
\label{hardyweight}
Let $f$ be holomorphic on the unit ball $\B^d$, represented by the Taylor 
expansion $f(z)= \sum_J a_J z^J$. Then $f \in H^2(\B^d)$ %(the Hardy space of the ball) 
if and only if 
$$
\sum_{J\in \N^d} \left(\frac{|a_J|}{\omega^{\B^d}_2(J)}\right)^2 < \infty, 
\mbox{ where } (\omega^{\B^d}_2(J))^{-2} = \frac{(d-1)! J!}{(|J|+d-1)!}.
$$
\end{lemma}
\begin{proof}
The surface measure $d\sigma$ on $\partial \B^d$ desintegrates as an integral of 
Haar measures on tori, so the monomials $(z^J)$ form an orthogonal system in $H^2(\B^d)$,
which is a basis since the polynomials are dense in the space. 
Then $\|f\|^2_{H^2(\B^d)}= \sum_J |a_J|^2 \|z^J\|^2_{H^2(\B^d)}$.
%We set  $\omega_2(J)^{-1}:= \|z^J\|_{H^2(\B^d)}$.  
The explicit value of $\|z^J\|_{H^2(\B^d)}=(\omega^{\B^d}_2(J))^{-1}$
(Definition \ref{stdweight}) can be found in 
\cite[p. 12]{RudinBall}.
\end{proof}

As an immediate consequence of this Lemma and of the remarks before Definition
\ref{stdweight}, if $X_{\omega, p}$ is  as in Definition \ref{series} and $p\ge 2$,
and if for all $J$, $\omega^\Omega_2(J)\le \omega(J)$, then $H^2(\Omega) \subset X_{\omega, p}$.

\begin{definition}
\label{gauge}
For $z \in \C^d$, let $|z|_{\D^d}:= \max_{1\le j \le d} |z_j|$, 
$|z|_{\B^d}^2:= \sum_{1\le j \le d} |z_j|^2$. 
\end{definition}
In each case, $\Omega=\{z: |z|_\Omega<1\}$.

%\begin{lemma}
%\label{convdom}
%Suppose the weight $\omega$ verifies the relative monotonicity condition \eqref{relmono}
%and that for any $ J \in \N^d$,
%\begin{equation}
%\label{relgrowth}
% \omega(J) \le  \omega^\Omega_2(J) \exp(o(|J|)),
%\end{equation}
%then for any $f \in X_{\omega, p}$, the series defining $f$ converges on $\Omega$; 
%and there is no larger domain on which every $f \in X_{\omega, p}$ has to be holomorphic.
%\end{lemma}

\begin{proof*}{\it Proof of Lemma \ref{lemzero}.}
The last statement follows from the fact that \eqref{relmono} implies that $H^2(\Omega) \subset X_{\omega, p}$,
as in the remark before the Definition.

To prove the convergence of the series,
 write $f(z)= \sum_J a_J z^J$ and take a point $z$ such that $|z|_\Omega=\rho<1$.
 Since $X_{\omega, p} \subset X_{\omega, \infty}$, $|a_J| \preceq \omega(J)$ and it will be enough
to prove the convergence of $\sum_J  \omega(J) |z^J|$.  

In the case where $\Omega=\D^d$, then
\[
\log ( \omega(J) |z^J|) \le |J| \log \rho + o(|J|) \le -\eta |J|
\]
for some $\eta>0$ when $|J|$ is large enough, so the general term
is dominated by the general term of a convergent geometric (multi)series.

In the case where $\Omega=\B^d$, for any $k\in\N$,
\[
|z|_\Omega^{2k}= \sum_{J: |J|=k} \frac{|J|!}{J!} |z^{2J}|.
\]
%Recall that for $J=(j_1, \dots, j_d)$, $|J|=j_1+\cdots +j_d$, $J!= j_1! \cdots j_d!$,
%$2J=(2j_1, \dots, 2j_d)$.

First consider only sums of terms with all powers even:
\begin{multline*}
\sum_{J: |J|=k} \omega(2J) |z^{2J}| \le \sup_{J: |J|=k} \left( \frac{\omega(2J) J!}{|J|!}\right)
\sum_{J: |J|=k} \frac{|J|!}{J!} |z^{2J}| 
\\
\le \sup_{J: |J|=k} \left( \frac{\omega(2J) J!}{|J|!}\right) \rho^{2k},
\end{multline*}
so using \eqref{weightcond}, we need to estimate $\frac{\omega_2(2J)^2 (J!)^2}{(|J|!)^2}$.

Stirling's formula implies that for any $n \in \N$, 
\[
\log (n!)= n(\log n -1) + o(n),
\]
so we have, for $|J|=k$,
\begin{multline*}
\log \left( \frac{\omega_2(2J)^2 (J!)^2}{(|J|!)^2} \right) =
\log \left( \frac{(2k + d-1)! (J!)^2}{(d-1)! (2J)! (k!)^2 } \right) =
\\
= (2k + d-1) \left( \log (2k + d-1) -1 \right) -2 k (\log k -1) 
\\
+ 2 \sum_{i=1}^d j_i ( \log j_i -1) - \sum_{i=1}^d 2 j_i ( \log (2j_i) -1)
+o(k)
\\
= 2k \log 2 - 2 \sum_{i=1}^d j_i \log2 + 2k \left( \log (k + \frac{d-1}2) -\log k \right) +o(k)
\\
= 0+ O(1) + o(k) = o(k).
\end{multline*}

This proves that $\sum_{J: |J|=k} \omega(2J) |z^{2J}|$ is dominated by the general
term of a convergent geometric series for $k$ large enough. 

Now consider a general $J$ such that $|J|=k$: then $J=2J'+K$, with $j'_i= 2 \left[j_i/2\right]$,
and $K \in \{ 0,1\}^d$. Let $2J'':= J+K$. 
Condition \eqref{relmono} shows that $\omega(J) \asymp \omega(2J') \asymp \omega(2J'')$, 
$|z^J| \le |z^{2J'}|$,
and each $J'$ corresponds to at most $2^d$ different multi-indices $J$. So 
$\sum_{J: |J|=k} \omega(J) |z^{J}|$ is dominated by the general
term of a convergent geometric series for $k$ large enough.

Continuity of the evaluation map follows, for instance, from the Dominated Convergence Theorem applied
to the series.
\end{proof*}

\begin{lemma}
\label{factorbound}
Suppose that $H^\infty(\Omega)$ is a multiplier space for $X$; or that $X=X_{\omega, p}$, with  $p\ge 2$ and
$\omega$ verifying \eqref{relmono}. Let $g\in H^\infty(\Omega)$, $K \in \N^d$. 
Then 
\[
\|z^K g \|_{X_{\omega, p}} 
\le \frac{C_m }{\omega(K)} \|g\|_{H^\infty(\Omega)}.
\]
% \le \frac{C_m }{\omega(K)} \|g\|_{H^2(\Omega)}
\end{lemma}

Observe that in the special case $K=0$, $X=X_{\omega, p}$, we get back the fact that $H^2(\Omega) \subset X_{\omega, p}$.
\begin{proof}
Under the first assumption, we immediately have
$$
\|z^K g \|^p_{X} \le \|z^K\|^p_{X_{\omega, p}}\| g \|_{H^\infty(\Omega)} = \frac{1 }{\omega(K)} \|g\|_{H^\infty(\Omega)}.
$$

Under the second assumption, by scaling we may assume $1=\|g\|_{H^2(\Omega)}\le  \|g\|_{H^\infty(\Omega)}$. 
%Then \|g \|_{X_{\omega, p}} \le \|g\|_{H^2(\B^d)} \le \|g\|_{H^\infty(\B^d)} =1$.
Let $g(z)= \sum_J a_J z^J$. Then $\sup_J \frac{|a_J|^2}{\omega_2 (J)^2}\le 1$.
%Since $X_{\omega, p} \subset H^2(\B^d) \subset H^\infty(\B^d)$,
%there exists $C>0$ independent of $g$ such that $|a_J| \le C \omega(J)$. 
Then 
\begin{multline*}
\|z^K g \|^p_{X_{\omega, p}} = \sum_J \frac{|a_J|^p}{\omega (J+K)^p} 
\le \left( \sup_J \frac{\omega_2 (J)}{\omega (J+K)} \right)^p 
%\left( \sup_J \frac{|a_J|}{\omega_2 (J)} \right)^{p-2} 
\sum_J \frac{|a_J|^p}{\omega_2 (J)^p}
\\
\le \frac{C_m^p}{ \omega (K)^p}  \sum_J \frac{|a_J|^2}{\omega_2 (J)^2} 
=
\frac{C_m^p}{ \omega (K)^p} \le   \frac{C_m^p}{ \omega (K)^p}\|g\|^p_{H^\infty(\Omega)}.
\end{multline*}
\end{proof}

\section{Auxiliary results}
\label{auxil}

\subsection{Multiplier property.}

\begin{definition}
\label{multipliers}
We shall say that $H^\infty(\Omega)$ is a \emph{multiplier algebra} for $X$ if there exists $C_m>0$ such that
$$
\forall f \in X, \forall g \in H^\infty(\Omega), gf \in X \mbox{ and }
\|gf\|_X \le C_m \|g\|_\infty \|f\|_X.
$$
\end{definition}
Notice that, since constants are in $X$, this implies that $H^\infty(\Omega) \subset X$.

It is immediate that $H^\infty(\Omega)$ is a multiplier algebra for each 
${\mathcal{B}}^{p}(\lambda)(\Omega)$, with $C_m=1$.  In the case of $X_{\omega, p}$, 
writing $\omega^\Omega_\infty(I):=\|z^I\|_{L^\infty(\Omega)}^{-1}$, an obvious necessary
condition is that 
\begin{equation}
\label{almostmono}
C_m \omega (I+J) \ge \omega^\Omega_\infty(I) \omega(J),
\end{equation}
but sufficient conditions are not so easy to state in general.

Observe that \eqref{almostmono} is very similar to \eqref{relmono}. In fact,
%when $\Omega={\D^d}$, 
$\omega^{\D^d}_\infty(I)=\omega^{\D^d}_2(I)=1$ for all $I$, and %when $\Omega={\B^d}$, 
one can show that 
\[
\omega^{\B^d}_2(I)\ge \omega^{\B^d}_\infty(I) \ge \omega^{\B^d}_2(I) - O(\log |I|),
\]
by an appropriate minoration of $|z^I|$ on a strip of $\partial {\B^d}$ of width comparable to $|I|$
around its maximum modulus set (we omit the details; this can provide an alternate
proof of Lemma \ref{lemzero} without recourse to Stirling's formula). 

\subsection{Some tools.}
Our first technical tool is a bound from below for the modulus of a zero-free bounded holomorphic function. 

For $z\in \Omega$, $z^*:= z/|z|_\Omega\in \partial \Omega$, where $|z|_\Omega$
is as in Definition \ref{gauge}.
\begin{lemma}
\label{harnack}
Let $U$ be a zero-free holomorphic function on $\Omega$ 
such that $ {\left\Vert{U}\right\Vert}_{\infty}\leq 1,$ and $z\in \Omega$.
Let $c^{2}:= \log \frac{1}{\left\vert{U(0)}\right\vert }.$
Suppose $k\ge 4c^2$.
 Then, for $\Omega=\D^d$,
 $$ \left\vert{U(z)}\right\vert +|z_1|^k +\cdots+ |z_d|^k \geq e^{-2c{\sqrt{k}}},
 $$ 
and for $\Omega=\B^d$,
$$
\left\vert{U(z)}\right\vert + \sum_{|J|=k}  |f_J(z)| \ge e^{-2c\sqrt {2k}},
$$
where $f_J(z):=p_J z^{2J} = \frac{|J|!}{J!} z^{2J}$.
\end{lemma}
\begin{proof}
The conclusion is obvious if $z=0$. If not,
define a holomorphic function on $\D$ by
 $f_{z^{*}}(\zeta ):=U(\zeta z^{*})$.
Then
\begin{itemize}
\item
$\ {\left\Vert{f_{z^{*}}}\right\Vert}_{\infty }\leq 1$;
\item
$f_{z^{*}}(0)=U(0)$;
\item
$\forall \zeta \in {\mathbb{D}},\ f_{z^{*}}(\zeta )\neq 0$;
\item
$ f_{z^{*}}(|z|_\Omega)=U(z).$
\end{itemize}
The Harnack inequality applied to the positive harmonic function
 $ \log \ \left\vert{f_{z^{*}}}\right\vert ^{-1}$ shows that
$$
 \left\vert{f_{z^{*}}(\zeta )}\right\vert \geq 
 \exp  {\left({-\frac{1+\left\vert{\zeta }\right\vert }{1-\left\vert{\zeta}\right\vert }\log \ \frac{1}{\left\vert{U(0)}\right\vert }}\right)}
 \ge 
\exp 
 {\left({-\frac{2 }{1-\left\vert{\zeta}\right\vert }\log \ \frac{1}{\left\vert{U(0)}\right\vert }}\right)} .
 $$

The computation implicit at the beginning of the proof of \cite[Lemma 3]{Seip11}
shows that $\inf_{\D} |f_{z^{*}}(\zeta)| + |\zeta|^k \ge e^{-2c\sqrt k}$
as soon as $k\ge 4c^2$; applying this 
to $\zeta = |z|_\Omega$,  we find
$$
\left\vert{U(z)}\right\vert +|z|_\Omega^k \geq  f_{z^{*}}(|z|_{\Omega}) + |z|_{\Omega}^k \geq e^{-2c\sqrt k}.
$$
In the case where $\Omega=\D^d$, this yields
$$
\left\vert{U(z)}\right\vert +|z_1|^k +\cdots+ |z_d|^k \geq  
\left\vert{U(z)}\right\vert +(\max_{1\le j \le d} |z_j| )^k \ge e^{-2c\sqrt k}.
$$
In the case where $\Omega=\B^d$, 
%we define
%\begin{equation}
%\label{deffj}
%f_J(z):=p_J z^{2J} = \frac{|J|!}{J!} z^{2J} .
%\end{equation}
substituting $2k$ for $k$, we obtain
$$
\left\vert{U(z)}\right\vert +(|z_1|^2 +\cdots+ |z_d|^2)^k =  
\left\vert{U(z)}\right\vert + \sum_{|J|=k}  |f_J(z)| \ge e^{-2c\sqrt {2k}}.
$$
\end{proof}

\begin{lemma}
\label{density}
If $X=X_{\omega, p} $  or ${\mathcal{B}}^{p}(\lambda )$ from Definitions \ref{series} or \ref{bergman}
or \ref{bergmanball}
respectively, then the space of polynomials $\C[Z]:=\C[Z_1, \dots, Z_d]$ is dense 
in $X$.
\end{lemma}
\begin{proof}
By construction, the polynomials are dense in $X_{\omega, p} $. 

For $f\in \mathcal{B}^{p}(\lambda)(\D^d)$, $r=(r_1,\dots,r_d) \in [0,\infty)^d$, let
$$
m_{r,p}(f) := \int_{\mathbb T^{d}}
 \left\vert f(r_1 e^{i\theta_1},\dots,r_d e^{i\theta_d}) \right\vert^p d\theta 
$$
denote the mean of $|f|^p$ on the torus $\mathbb T(r )$ of multiradius $r$. Since $|f|^p$ is plurisubharmonic,
this is an increasing function with respect to each component of $r$. In particular, if we set
for any $\gamma \in (0,1)$, $f_\gamma (z):= f(\gamma z)$, $m_{r,p}(f_\gamma) \le m_{r,p}(f)$
for each $r$.

We claim that $\lim_{\gamma \to 1} \|f-f_\gamma\|_{\mathcal{B}^{p}(\lambda)} =0$. Indeed,
$\|f-f_\gamma\|_{\mathcal{B}^{p}(\lambda)} = \int_{[0,1)^d} F_\gamma(r ) d\lambda (r) $,
where
$$
F_\gamma(r ) := \int_{\mathbb T^{d}}
 \left\vert f(r_1 e^{i\theta_1},\dots,r_d e^{i\theta_d})
 -f_\gamma (r_1 e^{i\theta_1},\dots,r_d e^{i\theta_d}) \right\vert^p d\theta .
$$
Since $|f-f_\gamma|^p \le C_p (|f|^p+|f_\gamma|^p)$,
$$
F_\gamma(r ) \le C_p ( m_{r,p}(f) + m_{r,p}(f_\gamma) ) \le 2 C_p  m_{r,p}(f)  \in L^1 (d\lambda).
$$
Since $f_\gamma \to f$ uniformly on the torus $\mathbb T ( r )$ for each $r$ as $\gamma \to 1$, 
$F_\gamma(r )\to 0$ for each $r$, and we can apply
Lebesgue's Dominated Convergence theorem. 

For each $\gamma \in (0,1)$, $f_\gamma$ is holomorphic on a larger polydisc, so can
be uniformly approximated by truncating its Taylor series. 

When $\Omega=\B^d$, we can perform an analogous (and simpler) argument.
\end{proof}

%\begin{lemma}
%\label{compnorm}
%For any $p\ge 2$, and any $\omega$ satisfying \eqref{almostmono}, $\|f\|_{\omega,p} \le C \|f\|_{H^2}$.
%\end{lemma}
%\begin{proof}
%$$
%\|f\|_{\omega,p}^{p}=\sum_{J}{\frac{\left\vert{a_{J}}\right\vert
% ^{p}}{\omega (J)^{p}}}
% \leq 
% \frac{C_m^{p}}{\omega (0)^{p}}\sum_{J}{\left\vert{a_{J}}\right\vert^{p}}
% \leq 
%\frac{C_m^{p}}{\omega (0)^{p}}(\sum_{J}{\left\vert{a_{J}}\right\vert^{2}})^{p/2} = 
%\frac{C_m^{p}}{\omega (0)^{p}}\|f\|_{H^2}^p.
%$$
%\end{proof}

\subsection{First reduction.}

We begin by showing that it is enough to obtain a relaxed version of the conclusion.
\begin{lemma}
\label{reduction}
Let $U\in H^\infty(\Omega)$ be a non-vanishing function.

If either:
\begin{itemize}
\item 
(i) $H^\infty(\Omega)$ is a multiplier algebra for $X$,
\item
or (ii) $X=X_{\omega,p}$, $p\ge 2$ and \eqref{relmono} is satisfied,
\end{itemize}
and if there exists a sequence $(f_n) \subset H^\infty(\Omega)$
such that 
$$\lim_{n\to\infty} \|1-f_nU\|_X=0,
$$
then $U$ is cyclic in $X$.
%there exists a sequence $(P_n) \subset \C[Z]$ such that $\lim_{n\to\infty}\|1-P_nU\|_X=0$.
\end{lemma}
\begin{proof}
By Lemma \ref{density}, it is enough to show that we can approximate 
any polynomial $P$. 

Let us show that it is enough to prove that for any $\varepsilon>0$, there exists $Q\in \C[Z]$ 
such that $\|1-Q U\|_X \le \eps$.

Let $P(z):= \sum_{|J|\le N} a_J z^J$, then
$$
\|P-PQU\|_X =\|P(1-QU)\|_X \le \|P\|_\infty \|1-Q U\|_X
$$
in the case of assumption (i), and
$$
\|P(1-QU)\|_X \le  \sum_{|J|\le N} |a_J| \|z^J(1-QU)\|_X
\le C_m  \left( \sum_{|J|\le N} \frac{|a_J|}{\omega_2^\Omega(J)} \right)\|1-QU\|_X,
$$
in the case of assumption (ii), and each upper bound
can be made arbitrarily small by choosing $Q$. 

In the case of assumption (i), let us then show that the constant function $1$ can be approximated. 
Let $\eps>0$. Take $f\in H^\infty(\D^d)$ such that $\|1-fU\|_X<\eps/2$. 
By Lemma \ref{density}, we can choose $Q\in \C[Z]$ such that $\|f-Q\|_X \le
\frac1{ C_m\|U\|_\infty } \frac\eps2$, where $C_m$ is as in Definition \ref{multipliers}.
Then
$$
\|1-Q U\|_X \le \|1-fU\|_X + \|U(f-Q)\|_X < \frac\eps2 + C_m \|U\|_\infty  \|f-Q\|_X \le \eps.
$$

In the case of assumption (ii), again take $f$ so that $\|1-fU\|_{\omega,p}$ is small, then
because $H^2(\Omega) \subset X_{\omega, p}$,
%by Lemma \ref{factorbound} applied with $K=0$,
$$
\|fU-QU\|_{\omega,p} \le C \|fU-QU\|_{H^2} \le C \|U\|_\infty \|f-Q\|_{H^2} ,
$$
and this last quantity can be made arbitrarily small by taking $Q$ a Taylor expansion
of $f$ for instance.
\end{proof}

\section{Proof of Theorem \ref{exproot}}
\label{proof1}

\begin{proof*}{\it Proof of Theorem \ref{exproot}.}

Observe that if $X= {\mathcal{B}}^{p}(\lambda )$, then $H^\infty(\Omega)$ is a multiplier algebra, so Lemma \ref{factorbound}
always applies here.

{\bf Case 1: $\Omega=\D^d$.}

Let $c^{2}:=-\log \ \left\vert{U(0)}\right\vert$ and $B >2c(2d+1)$. 
By the hypothesis of the theorem,
there exists a strictly increasing sequence $(n_k)_{k\ge1}$ such that
for all $k$, $\log \tilde \omega(n_k) \ge B \sqrt{n_k}$.
 
By Lemma \ref{harnack} and Theorem \ref{corona}, we get 
$g_{j} \in H^{\infty }({\mathbb{D}}^{d})$, for $j=1,\dots,d+1$, such that
%for any $z\in \D^d$,
\begin{equation}
\label{bezout}
 g_{d+1}U+g_{1}z_{1}^{n_k}+ \cdots +g_{d}z_{d}^{n_k}=1,
\end{equation}
and 
\begin{equation}
\label{corest}
 \forall j=1, \dots,\ d+1,\ {\left\Vert{g_{j}}\right\Vert}_\infty
 \leq C(d)n_k^{d}e^{2c(2d+1){\sqrt{n_k}}}.
\end{equation}
Set $f_{k}:=g_{d+1},$ we get,
%$$
% {\left\Vert{1-f_{k}U}\right\Vert}_{\omega ,2}\leq
% \sum_{j=1}^{d}{{\left\Vert{g_{j}z_{j}^{n_k}}\right\Vert}_{\omega ,2}}
%$$
%and using the fact that $H^\infty(\D^d)$ is a multiplier algebra for $X$, 
 using Lemma \ref{factorbound}, 
\begin{equation}
\label{1fU}
{\left\Vert{1-f_{k}U}\right\Vert}_X \leq
 \sum_{j=1}^{d}{{\left\Vert{g_{j}z_{j}^{n_k}}\right\Vert}_X}
 \leq 
 C_m \sum_{j=1}^{d}\frac{\left\Vert{g_{j}}\right\Vert_{\infty}}{\omega(n_k e_j)}
 \leq C_m \frac{C(d)n_k^{d}e^{2c(2d+1){\sqrt{n_k}}}}{\tilde \omega(n_k)}.
\end{equation}

By the choice of $B$, this tends to $0$ as $k\to\infty$.
It only remains to apply lemma \ref{reduction} to conclude. 
\smallskip

{\bf Case 2: $\Omega=\B^d$.}

Let $c$, $(n_k)$ be as above and  $B >2c \sqrt2 (2d+1)$.

By Theorem \ref{corona}, we will get 
$g_0, g_{J} \in H^{\infty }({\mathbb{D}}^{d})$, for $|J|=n_k$, such that
%for any $z\in \D^d$,
\begin{equation}
\label{bezoutB}
 g_{0}U+ \sum_{|J|=n_k} g_J f_J \equiv 1,
\end{equation}
where $f_J$ is as in Lemma \ref{harnack}. 

We need to estimate the size of the $g_J, g_0$.  

First $\sum_{|J|=n_k} |f_J(z)| =|z|_\Omega^{2n_k} <1$.

The number of terms in the Bezout equation is 
$$
N=N(d,k)= \# \left\{ J \in \N^d: |J|=n_k\right\}= \frac{(n_k+d-1)!}{n_k! (d-1)!} \le n_k^{d-1}.
$$
We also need 
 ${\left\Vert{\nabla f_{J}}\right\Vert}_{\infty }.$ 
 We have
\[
 \frac{\partial }{\partial z_{i}}f_{J}=p_{J}2j_{i}\frac{z^{2J}}{z_{i}}\Rightarrow
 \nabla f_{J}=p_{J}z^{2J}(\frac{2j_{1}}{z_{1}},...,\frac{2j_{d}}{z_{d}})
 =2f_{J}(z)(\frac{j_{1}}{z_{1}},...,\frac{j_{d}}{z_{d}}).
 \] 
If we set $\tilde J:=(\max (0,2j_{1}-1),...,\max (0,2j_{d}-1))$
 and 
\newline
$\tilde z_{i}:=z_{1}\cdot \cdot \cdot z_{i-1}\hat z_{i}z_{i+1}\cdot
 \cdot \cdot z_{d},$ where $\displaystyle \hat z_{i}$ is omitted,
 then 
 \[
 \nabla f_{J}(z)=2p_{J}z^{\tilde J}(j_{1}\tilde
 z_{1},...,j_{d}\tilde z_{d}).
 \]
 So we get, because $\left\vert{\tilde z_{i}}\right\vert \leq 1$ in the ball,
 \[
 \left\vert{\nabla f_{J}(z)}\right\vert \leq 2p_{J}\left\vert{z^{\tilde
 J}}\right\vert \sum_{i=1}^{d}{j_{i}\left\vert{\tilde z_{i}}\right\vert
 }\leq 2p_{J}\left\vert{z^{\tilde J}}\right\vert \left\vert{J}\right\vert
 .
 \]
But if we write $J':= ((j_1-1)_+,\dots,(j_d-1)_+)$, then $|z^{\tilde J}| \le |z^{2J'}|$
and $\sum_{|J'|=n_k-d} p_{J'} |z^{2J'}|  <1$. Furthermore, 
\[
p_J \le \frac{n_k(n_k-1) \cdots (n_k-d+1)}{j_1 \cdots j_d} p_{J'} \le n_k^d p_{J'}.
\]  
All together then, $\|\nabla f_{J}(z)\|_\infty \le C(d) n_k^{d+1}$.

By Lemma \ref{harnack} (in the case of the ball), 
$$
\delta = \inf_{z\in \B^d} \left( |U(z)| + \sum_{|J|=n_k} |f_J(z)| \right) \ge e^{-2c\sqrt {2n_k}}.
$$ 
Gathering the estimates, we get 
\begin{equation}
\label{corestB}
\|g_J\|_\infty \le C(d) N(d,k)^{4d+2} e^{2c (2d+1)\sqrt {2n_k}} \le C(d) n_k^{5d^2} e^{2c (2d+1)\sqrt {2n_k}}.
\end{equation}
Then let $f_k=g_0$ (at the $n_k$ step)
\begin{multline}
\label{1fUB}
{\left\Vert{1-f_{k}U}\right\Vert}_X \leq
 \sum_{|J|=n_k} \|g_J f_J\|_X \le C_m \sum_{|J|=n_k} p_J \|z^{2J}\|_X \|g_J \|_\infty 
 \\
 \leq C_m \frac{C(d)n_k^{5d^2}e^{2c(2d+1){\sqrt {2n_k}}}}{\tilde \omega(n_k)},
\end{multline}
and we finish as before.
\end{proof*}

\section{Proof of Theorem \ref{mainthm}}
\label{proof2}

\subsection{Main intermediate result.}

\begin{theorem}
\label{metathm} 
Let $X$ be a Banach space as in Definitions \ref{series}, \ref{bergmanball} or \ref{bergman}. Suppose that 
$H^\infty(\D^d)$ is a multiplier algebra for $X$. Suppose also that 
$\lim_{k\to\infty} \tilde \omega (k) = \infty$, that $\log \tilde \omega (k) = o(k)$,
and that conditions  \eqref{relmono}, \eqref{weightcond} and \eqref{div2} hold.

Then any $U \in H^\infty(\D^d)$, verifying $U(z)\neq 0$
for any $z \in \D^d$ is cyclic in $X$.
%there exists a sequence $(P_n) \subset \C[Z]$ such that $\lim_{n\to\infty}\|1-P_nU\|_X=0$. 
\end{theorem}
\begin{proof}

Now we need to distinguish two cases according to the growth of $\omega(k)$.

{\bf Case 1:} $\sup_k \frac{\log \tilde \omega(k)}{\sqrt k} = \infty$. 

Then Theorem \ref{exproot} applies.

%We perform the same proof as in the proof of Theorem \ref{exproot},
%using the fact that $H^\infty(\D^d)$ is a multiplier algebra for $X$,
%instead of Lemma \ref{factorbound}, to get directly 
%$\|z^J g\|_X \le \frac{C_m}{\omega (J)} \|g\|_\infty$.

\smallskip

{\bf Case 2:} $\sup_k \frac{\log \tilde \omega(k)}{\sqrt k} = B < \infty$. 
To deal with this more delicate case, we shall need the full power of the proof scheme 
in \cite{Seip11}. Since our Corona-like estimates are slightly different from those
in dimension $1$, we first need a refined version of \cite[Lemma 1]{Seip11}.

\begin{lemma}
\label{subseq}
Let $\tilde \omega$ be as in Theorem \ref{metathm}. Let $C_0>0$. Then there exists
a strictly increasing sequence $(n_k)_{k\ge1}$ such that 
$$
\sum_{k\ge 1}  \frac{\left(\log \tilde \omega (n_k)\right)^2}{n_k}  = \infty,
$$
and for all $k$,
$\log \tilde \omega (n_{k+1}) \ge 2 \log \tilde \omega (n_k)$ 
and $\log \tilde \omega (n_k) \ge C_0 \log n_k$.
\end{lemma}
The last condition is the only novelty with respect to \cite[Lemma 1]{Seip11}.
\begin{proof}
First notice that there exists an infinite set $E \subset \N^*$ such that
for all $n\in E$, $\log \tilde \omega (n) \ge C_0 \log n$.
Indeed, if not, for $n$ large enough, we would have 
$$
\log \tilde \omega (n) \le C_0 \log n \le n^{1/4},
$$
and \eqref{div2} would be violated.

Now let $n_0=1$ and if $n_j$ is defined, let
\begin{multline*}
n'_{j+1} := \min \left\{ n>n_j : \log \tilde \omega (n) \ge 2 \log \tilde \omega (n_j) \right\},
\\
n_{j+1} := \min \left\{ n>n_j, n\in E : \log \tilde \omega (n) \ge 2 \log \tilde \omega (n_j) \right\}.
\end{multline*}
Obviously, $n_j < n'_{j+1}  \le n_{j+1}$. We claim that 
$$
S:= \sum_{j\ge0} \sum_{k=n'_j}^{n_j-1} \left( \frac{\log \tilde \omega (k)}{k} \right)^2 <\infty.
$$
Accepting the claim, the proof finishes as in \cite{Seip11}:
\begin{multline*}
\sum_{k\ge 1} \left( \frac{\log \tilde \omega (k)}{k} \right)^2
 \le S+ \sum_{j\ge0} \sum_{k=n_j}^{n'_{j+1}-1} \left( \frac{\log \tilde \omega (k)}{k} \right)^2
 \\
\le S+ \sum_{j\ge0} 4 \left(\log \tilde \omega (n_j)\right)^2 \sum_{k=n_j}^{n'_{j+1}-1} \frac1{k^2}
\le S+ 4 \sum_{j\ge0} \frac{\left(\log \tilde \omega (n_j)\right)^2}{n_j-1},
\end{multline*}
so the last sum must diverge.

We now prove the claim. If $n'_j \le k <n_{j}$, then $n \notin E$, so for $j$ large enough and 
$n'_j \le k <n_{j}$, $\log \tilde \omega (k) \le k^{1/4}$, thus
\begin{equation}
\label{resid}
\sum_{k=n'_j}^{n_j-1} \left( \frac{\log \tilde \omega (k)}{k} \right)^2 
\le \sum_{k\ge n'_j} \frac1{k^{3/2}} \le \frac2{\sqrt{n'_j}-1}.
\end{equation}
The definition of $n_j$ implies that $\tilde \omega (n_j) \ge C_0 2^j$, and $n'_{j+1}  \notin E$ (if it 
is distinct from $n_{j+1}$) so
$$
C_0 \log n'_{j+1} > \log \tilde \omega (n'_{j+1}) \ge 2 \log \tilde \omega (n_j) \ge C_0 2^j,
$$
and the series with general term the last expression in \eqref{resid} must converge. 
\end{proof}
We follow the proof of \cite[Theorem 1]{Seip11}, with a couple of wrinkles.

Choose $A:= \max(2, \log C(d))$ where $C(d)$ is the constant in \eqref{corest} when
$\Omega=\D^d$ (resp. \eqref{corestB} when $\Omega=\B^d$). 
Then  for $c^2=-\log|U(0)|$ as above, let 
$C_1^2:= (8\sqrt2 (2d+1) A c)^2 +B^2$. We choose $C_0 \ge C_1/c$ when
$\Omega=\D^d$ (resp. $C_0 \ge \frac{5d^2 C_1}{2\sqrt2 (2d+1) c}$ when $\Omega=\B^d$), and define the sequence $(n_j)$
as in Lemma \ref{subseq} above. For any given $j_0\in \N$, let
$\alpha_j^2:= \frac{\left(\log \tilde \omega (n_{j_0+j})\right)^2}{n_{j_0+j}} $,
$$
N:= \min \left\{ M: \sum_{j=1}^M \alpha_j^2 \ge (8\sqrt2 (2d+1) A c)^2 \right\},
$$
$\lambda_j:= \alpha_j \left( \sum_{i=1}^N \alpha_i^2\right)^{-1/2}$. 

Notice that for any $j$, $\alpha_j \le B$ by the hypothesis of Case 2, and that
$$
\sum_{i=1}^N \alpha_i^2 \le \sum_{i=1}^{N-1} \alpha_i^2 + \alpha_N^2 \le (8\sqrt2 (2d+1) A c)^2 +B^2 = C_1^2,
$$
so that $\lambda_j \ge \alpha_j/C_1$. Clearly, $\lambda_j \le \alpha_j/(8\sqrt2 (2d+1) A c)$.

We write $U_j:=U^{\lambda_j^2}$, so that $U= \prod_{j=1}^N U_j$.
As above, choose $f_{j}:=g_{d+1}$  satisfying \eqref{bezout} and \eqref{corest}, but with
$U_j$ instead of $U$ and $n_{j_0+j}$ instead of $n_k$. 
The quantity $c$ must then be replaced by $c \lambda_j$. 

When $\Omega=\D^d$, the bound \eqref{corest} can be rewritten
\begin{equation}
\label{fjbd}
\|f_j\|_\infty \le \exp \left( 2 c (2d+1) \lambda_j \sqrt{n_{j_0+j}} + d \log n_{j_0+j} + \log C(d) \right).
\end{equation}
Notice that 
$$
c \lambda_j \sqrt{n_{j_0+j}} \ge \frac{c}{C_1}\log \tilde \omega (n_{j_0+j})
\ge  \frac{c C_0}{C_1}\log n_{j_0+j} \ge \log n_{j_0+j}, 
$$
by our choice of $C_0$, so that 
\begin{equation}
\label{fjbdA}
\|f_j\|_\infty \le \exp A \left( 2 c (2d+1) \lambda_j \sqrt{n_{j_0+j}} + 1 \right).
\end{equation}

We finish as in \cite{Seip11}. Let $f:= \prod_{j=1}^N f_j$. Since
\begin{multline*}
1-fU = 1- \prod_{j=1}^N f_j U_j = \sum_{k=1}^N (1-U_kf_k) \prod_{j=1}^{k-1} f_j U_j,
\\
\|1-fU \|_X \le C_m \sum_{k=1}^N \|1-U_kf_k\|_X \prod_{j=1}^{k-1} \|f_j U_j\|_\infty,
\end{multline*}
which by \eqref{1fU} becomes
\begin{multline}
\label{sumom}
\le C_m \sum_{k=1}^N \frac{C(d)n_{j_0+k}^{d}e^{2c(2d+1){\sqrt{n_{j_0+k}}}}}{\tilde \omega(n_{j_0+k})} \prod_{j=1}^{k-1} \|f_j \|_\infty
\\
\le 
C_m \sum_{k=1}^N \frac1{\tilde \omega(n_{j_0+k})} 
\exp \left( A \sum_{j=1}^k (2 c (2d+1) \lambda_j \sqrt{n_{j_0+j}} + 1) \right),
\end{multline}
and using the growth of $\log \tilde \omega(n_{j}) $ obtained in Lemma \ref{subseq},
\begin{multline*}
\le 
C_m \sum_{k=1}^N \frac1{\tilde \omega(n_{j_0+k})} 
\exp \left( Ak+  \sum_{j=1}^k \frac1{4} \log \tilde \omega(n_{j_0+j}) \right)
\\
\le
C_m \sum_{k=1}^N \exp \left( Ak - \frac1{2} \log \tilde \omega(n_{j_0+k}) \right).
\end{multline*}
Now choose $j_0$ such that $\log \tilde \omega(n_{j_0}) \ge A$, the sum above
has terms with better than geometric decrease, so is bounded by $\tilde \omega(n_{j_0+1})^{-1/2}$,
which can be made arbitrarily small by choosing $j_0$ large enough.

When $\Omega=\B^d$, we need to make the changes indicated at the beginning of the argument,
and replace the bound \eqref{fjbd} by the following:
\[
\|f_j\|_\infty \le \exp \left( 2 c \sqrt2 (2d+1) \lambda_j \sqrt{n_{j_0+j}} + 5d^2 \log n_{j_0+j} + \log C(d) \right).
\]
Then the choice (for $\Omega=\B^d$) of $C_0$ implies that 
$2 c \sqrt2 (2d+1) \lambda_j \sqrt{n_{j_0+j}} \ge 5d^2 \log n_{j_0+j}$, and this leads again to \eqref{fjbdA}.
In the succession of majorations that follow, \eqref{sumom} becomes
\begin{multline*}
\le C_m \sum_{k=1}^N \frac{C(d)n_{j_0+k}^{5d^2}e^{2c\sqrt2 (2d+1){\sqrt{n_{j_0+k}}}}}{\tilde \omega(n_{j_0+k})} \prod_{j=1}^{k-1} \|f_j \|_\infty
\\
\le 
C_m \sum_{k=1}^N \frac1{\tilde \omega(n_{j_0+k})} 
\exp \left( A \sum_{j=1}^k (2 c (2d+1) \lambda_j \sqrt{n_{j_0+j}} + 1) \right),
\end{multline*}
and the proof concludes in the same way.
\end{proof}

\subsection{Proof of Theorem \ref{mainthm}.}

We now obtain cyclicity results as soon as we can prove that $H^\infty(\D^d)$ 
is a multiplier algebra on the 
space $X$. As remarked after Definition \ref{multipliers}, this is always the case when
$X={\mathcal{B}}^{p}(\lambda)$. So we obtain Theorem \ref{mainthm} (i).

When  $X=X_{\omega,2}$ and $\omega$ is relatively nondecreasing, then
Lemma \ref{relabs} reduces us to the nondecreasing case,  
where the multiplication operators by each
$z_j$  are commuting contractions on a Hilbert space. Von Neumann's inequality
was generalized by Ando in the case of two contractions, and to an arbitrary number
of weighted shifts by Michael Hartz \cite{Ha}: this is precisely our situation.
It implies that for any polynomial $f$, and thus for any $f\in H^\infty(\D^d)$,
$\|f g\|_X \le \|f\|_\infty \|g\|_X$. So we obtain Theorem \ref{mainthm} (ii).

\section{Proof of the Corona theorem with smooth data}
\label{coronita}

We begin by constructing a partition of unity which exploits the smoothness of the data.

Because of the corona hypothesis, and $ f_{j}$
 is continuous up to the boundary of $\Omega$,
 for $1\le j\leq N-1,$ we have that  $
 g(z):=\sum_{j=1}^{N-1}{\left\vert{f_{j}(z)}\right\vert }$ is
 continuous in $\Omega$, and even Lipschitz with a constant controlled by 
 $\max_{1\le j \le N-1}\|\nabla f_j\|_\infty$.
  
Set 
$$
U'_{N}:=\lbrace z\in \bar \Omega :g(z)<\frac{N-1}{4N}\delta
 \rbrace
 \mbox{ and }
U_{N}:=\lbrace z\in \bar \Omega :g(z)<\frac{N-1}{2N}\delta
 \rbrace ,
 $$ 
and
$$
U_{j}:=\lbrace z\in \bar \Omega:\left\vert{f_{j}}\right\vert
 >\frac{\delta }{5N}\rbrace
 \mbox{ and }
U_{j}':=\lbrace z\in \bar \Omega:\left\vert{f_{j}}\right\vert >\frac{\delta
 }{4N}\rbrace .
 $$
Then $U'_{j}\Subset U_{j},$ $1 \le j \le N$.

\begin{lemma}
\label{partition}
There exist $C_1>0$ and 
$ \chi _{j}\in {\mathcal{C}}^{\infty }_{c}(U_{j})$, $j=1,...,N$,
such that for $z\in \Omega$, $0\leq \chi _{j}\leq 1,$ $\sum_{j=1}^{N}{\chi _{j}(z)}=1$,
%relative to the covering $\displaystyle \lbrace U_{j}\rbrace _{j=1,...,N}$ of $\displaystyle \bar \Omega $ 
and
\begin{multline*}
 \left\vert{\frac{\chi _{j}}{f_{j}}}\right\vert \leq \frac{C_1 N}{\delta }, \ 
 %\quad
 \|\nabla \chi_j \|_\infty \le \frac{C_1 N^2}\delta \max_{1\le i \le N-1}\|\nabla f_i\|_\infty, j=1,\dots,N, \ 
%  \|\nabla \chi_N \|_\infty \le \frac{C_1}\delta ,
\\
 \max_{1\le j \le N} \sup_{z\in \Omega} \frac{|\nabla \chi_j(z)|}{|f_j(z)|}
 \le C_1 \frac{ N^3}{\delta^2 } \max_{1\le i \le N-1}\|\nabla f_i\|_\infty,
\end{multline*}
 where $C_1$ is an absolute constant.
\end{lemma}

\begin{proof}
We can construct a function $\psi _{N}\in  {\mathcal{C}}_{c}^{\infty }(U_{N})$ such that $0\le \psi_{N}\le 1$ and  $\psi_{N}\equiv 1$ on 
$ U'_{N}$, with $\|\nabla \psi _{N}\|_\infty \le \frac{C}\delta$, for instance
 by composing $|g|$ with an appropriate smooth one-variable function.

We have 
$$ {\mathcal{O}}:=\bar U'_{N}\cup \bigcup_{j=1}^{N-1}{U_{j}'}\supset
 \bar \Omega,
$$
because for $z\notin \bigcup_{j=1}^{N-1}{U_{j}'},$ then
$$
\forall j=1,...,N-1,\ \left\vert{f_{j}(z)}\right\vert
 \leq \frac{\delta }{4N}\Rightarrow \sum_{j=1}^{N-1}{\left\vert{f_{j}(z)}\right\vert
 }\leq \frac{N-1}{4N}\delta \Rightarrow z\in \bar U'_{N}.
 $$
Now we construct a partition of unity $\lbrace \chi_{j}\rbrace _{j=1,...,N}$ 
subordinated to $\lbrace U_{j}\rbrace $ in the usual way: we take a nonnegative function
$ \psi _{j}\in {\mathcal{C}}^{\infty }_{c}(U_{j})$ such that $\psi _{j}\le 1$
everywhere and $\psi_j\equiv 1$
 on $\displaystyle U_{j}'$, with $\|\nabla \psi _{j}\|_\infty \le C \frac{N}\delta \|\nabla f_{j}\|_\infty$. 
 We set
$$
\chi _{j}:=\frac{\psi _{j}}{\sum_{k=1}^{N}{\psi_{k}}}.
$$ 
Since $\sum_{k=1}^{N}{\psi _{k}}\geq 1,$  we
 have $ 0\leq \chi _{j}\leq 1$, $\chi _{j}\in
 {\mathcal{C}}^{\infty }_{c}(U_{j})$ and $\chi_{1}+\cdots +\chi _{N}=1$ on $ \bar \Omega$
 and 
 $$
\|\nabla \chi _{j}\|_\infty \le C \frac{N^2}\delta \max_{1\le i \le N-1}\|\nabla f_i\|_\infty. 
$$

This yields a partition of unity such that$\frac{\chi _{j}}{f_{j}} \in {\mathcal{C}}^{\infty }(\bar \Omega)$
for $1 \le j \le N$ and  for $ j\leq N-1$,
$\left\vert{\frac{\chi _{j}}{f_{j}}}\right\vert \leq \frac{5N}{\delta }$,
because $\supp \chi _{j}\subset U_{j}$, where 
 $\left\vert{f_{j}}\right\vert >\frac{\delta}{5N}$ and $\chi _{j}\leq 1.$
 
For $j=N$ on the other hand, we have $\supp \chi_{N}\subset U_{N}$ and, 
by the corona hypothesis, 
$$
 z\in U_{N}\Rightarrow \left\vert{f_{N}(z)}\right\vert
 \geq \delta -g(z)=\delta -\frac{N-1}{2N}\delta =\frac{N+1}{2N}\delta
 $$ 
hence 
$$\left\vert{\frac{\chi _{N}}{f_{N}}}\right\vert
 \leq \frac{2N}{(N+1)\delta }\leq \frac{5N}{\delta }.
$$
An analogous reasoning yields the bound on $\frac{|\nabla \chi_j|}{|f_j|} $, $1\le j \le N$. 
 \end{proof}

\begin{proof*}{\it Proof of Theorem \ref{corona}.}
We shall now go through the Koszul complex method, introduced in
 this context by H\"ormander \cite{HormCor67}, to obtain the explicit bounds we need.
We follow the notations of \cite{AmMenDiv00}.

Let $\wedge^{k}({\mathbb{C}}^{N})$ be the exterior
 algebra on $ {\mathbb{C}}^{N},$ let $
 e_{j},\ j=1,...,N,$ be the canonical basis of 
 $\wedge^{1}({\mathbb{C}}^{N}),$ and 
 $ e_{\alpha }:=e_{\alpha_{1}}\wedge \cdots \wedge e_{\alpha _{k}}$, 
 $\alpha_{j}\in \lbrace 1,\dots, N\rbrace,$ the associated basis of $
 \wedge^{k}({\mathbb{C}}^{N}).$
 
Let $L_{r}^{k}$ be the space of bounded and infinitely differentiable
differential
 forms in $\Omega $ of type $(0,r)$
 with values in $\wedge^{k}({\mathbb{C}}^{N}).$
The norm on these spaces is defined to be the maximum of the uniform norms of 
 the coefficients.
 
We define two linear operators on $L_{r}^{k}$.
$$ 
\forall \omega \in L_{r}^{k},\quad R_{f}(\omega )
:=\omega \wedge \sum_{j=1}^{N}\frac{\chi_j}{f_j}e_{j} \in L_{r}^{k+1}.
 $$
We see that $\|R_f \omega\| \le C_f \|\omega\|$, with 
\begin{equation}
\label{normR}
C_f:= N \sup_{1\le j\le N, z\in \Omega} \left| \frac{\chi_j(z)}{f_j(z)}\right|.
\end{equation}
The operator $d_{f}: L_{r}^{k+1} \longrightarrow L_{r}^{k}$ is defined by induction and
 linearity.
For $\omega \in L_{r}^{0}$, $d_{f} \omega =0$.
To define the operator on $L_{r}^{1}$, set $ d_{f}(e_{j}):=f_{j}$ and extend by linearity.

To define $d_f$ on $L_{r}^{k+1}$, for $e_{\alpha }\in \wedge^{k}({\mathbb{C}}^{N})$, $1\le j \le d$, set
$$
d_{f}(e_{\alpha }\wedge e_{j}):=f_{j}e_{\alpha }-d_{f}(e_{\alpha })\wedge e_{j}\in L_{r}^{k}.
$$
It follows that $\|d_f\|_{\mathcal L (L_{r}^{k+1},L_{r}^{k})} \le C(k) \max_{1\le j \le N}\|f_j\|_\infty$.

It is easily seen by induction that $d_{f}^{2}=0$,
 $\bar \partial d_{f}\omega =d_{f}\bar \partial \omega $ and
$$
d_{f}\omega =0\Rightarrow d_{f}(R_{f}\omega )=\omega ,
$$
i.e. $\lambda=R_{f}\omega $ is a solution to
 the equation $d_{f}\lambda =\omega $ 
 when the necessary condition $d_{f}\omega =0$ is verified.

Together with the operator $\bar \partial : L_{r}^{k} \longrightarrow L_{r+1}^{k} $, we
 have a double complex, whose elementary squares are commutative
 diagrams. 
 
 We now construct by induction, for $0\le k\le N$, forms $\omega_{k,l} \in L^k_l$
and $\alpha_{k,l} \in L^{k+1}_l$, where $l\le k \le l+1$. 

We start with $\omega_{0,0}=1$, 
$$
\omega _{1,0}:= R_f(\omega_{0,0})
= \sum_{j=1}^{N}{\frac{\chi_{j}}{f_{j}}e_{j}}\in L_{0}^{1}.
$$
Then, if $\omega_{k,k-1}$ is given, we set $\omega_{k,k}:= \bar \partial \omega_{k,k-1}$; if 
$\omega_{k,k}$ is given, we set $\omega_{k+1,k}:= R_f \omega_{k,k}$. This construction stops
for $k=d$ since there are no $(0,d+1)$ forms on $\C^d$. 

{\bf Claim.}
For any $k\ge 0$, $d_f \omega_{k+1,k} =\omega_{k,k}$.

We prove the claim by induction. It is enough to see that $d_f \omega_{k,k}=0$.  For $k=0$,
this is true by construction.  For $k\ge 1$, assume the property holds at rank $k-1$.  Then
$$
d_f \omega_{k,k}= d_f \bar \partial \omega_{k,k-1} = \bar \partial d_f \omega_{k,k-1}
= \bar \partial  \omega_{k-1,k-1}=\bar \partial^2  \omega_{k-1,k-2}=0.
$$

From the construction, we have $\|\omega_{k+1,k}\| \le  C_f \| \omega_{k,k}\|$,
with $C_f$ defined in \eqref{normR}.
Since 
$$
\omega_{k,k}= \bar \partial ( R_f\omega_{k-1,k-1}) = \bar \partial \left( \omega_{k-1,k-1} \wedge \sum_{j=1}^{N}\frac{\chi_j}{f_j}e_{j} \right) = 
\omega_{k-1,k-1} \wedge \bar \partial \left(  \sum_{j=1}^{N}\frac{\chi_j}{f_j}e_{j} \right)
$$
because $\omega_{k-1,k-1}$ is $\bar\partial$-exact, we find $\|\omega_{k,k}\| \le  D'_f  \| \omega_{k-1,k-1}\|$,
with 
$$
D'_f := N \sup_{1\le j\le N, z\in \Omega} \frac{\|\nabla\chi_j(z)\|}{|f_j(z)|}.
$$
By an immediate induction, $\|\omega_{k,k}\| \le  (D'_f)^k$, $\|\omega_{k+1,k}\| \le C_f (D'_f)^k$.

We proceed with the construction of the forms $\alpha_{k,l}$, by 
descending induction.  Set $\alpha_{d+2,d}=\alpha_{d+1,d}=0$. 
Since $\bar \partial\omega_{d+1,d}=0$ by degree reasons,
there exists $u \in L_{d-1}^{d+1}$ such that $\bar \partial u = \omega_{d+1,d}$, and
$\|u\| \le E_d \| \omega_{d+1,d}\|$. We set $\alpha_{d+1,d-1}=u$.

Suppose given $\alpha_{k+1,k}=d_f \alpha_{k+2,k}$, with $\bar\partial \omega_{k+1,k}- \bar\partial \alpha_{k+1,k}=0$
(this is trivially verified when $k=d$).  Then the hypothesis on $\Omega$ implies that there exists
$u \in L_{k-1}^{k+1}$ such that 
$$\|u\| \le E_k \| \omega_{k+1,k}-  \alpha_{k+1,k}\|\mbox{ and }
\bar \partial u = \omega_{k+1,k}-  \alpha_{k+1,k}. 
$$
Then we set $\alpha_{k+1,k-1}:=u$. 

Finally, we put $\alpha_{k,k-1}:= d_f \alpha_{k+1,k-1}$. We need to check the condition on $\bar \partial$:
$$
\bar \partial \alpha_{k,k-1} = d_f \bar \partial \alpha_{k+1,k-1} = 
d_f \left( \omega_{k+1,k}-  d_f \alpha_{k+2,k} \right) = d_f  \omega_{k+1,k} = \omega_{k,k} = \bar \partial \omega_{k,k-1}.
$$
The following
 diagram,  where ${\mathcal{S}}$ stands for the operator
 solving the $\bar \partial $ equation, 
 describes the whole complex for $n=2,$ $N=3$.

\resizebox{15cm}{!}{\includegraphics{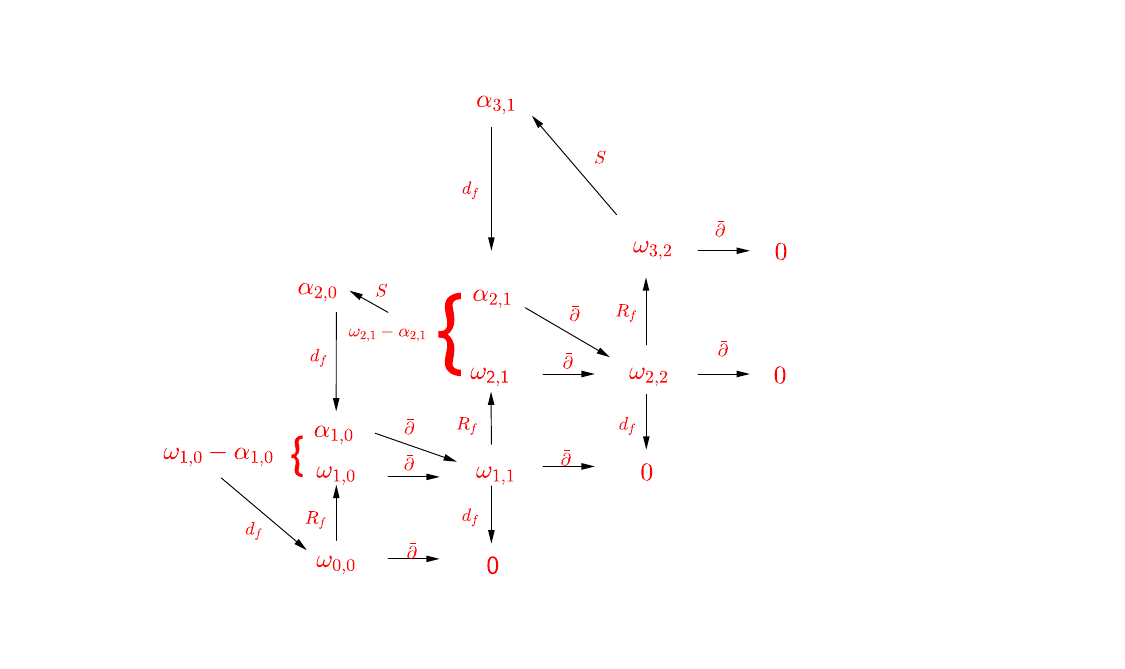}}

The bounds on the solution of the Cauchy-Riemann equation $\bar \partial$ 
and those on $\omega_{k,l}$ imply that 
$$
\|\alpha_{k+1,k-1}\| \le E_k \left( C_f (D'_f)^k + \|d_f\| \|\alpha_{k+2,k}\|\right),
$$
from which we deduce by induction
$$
\|\alpha_{k+1,k-1}\| \le C_f \sum_{j=k}^{d-1} (D'_f)^j \|d_f\|^{j-k} \prod_{i=k}^j E_j
+ \|d_f\|^{d} \left( \prod_{j=k}^{d-1} E_j \right) \|\alpha_{d+1,d-1}\|,
$$
so taking into account the bound $\|\alpha_{d+1,d-1}\| \le E_d  \|\omega_{d+1,d}\|$,
we have for any $k$
$$
\|\alpha_{k+1,k-1}\| \le C(d) \|f\|_\infty^d \left( \prod_{j=1}^{d} E_j \right) C_f (D'_f)^{d}.
$$
Finally, we claim that a solution to the Bezout equation is given by
the components of 
$\gamma_{1,0}:=\omega _{1,0}-\alpha _{1,0}=: \sum_{j=1}^{N}{g_{j}e_{j}}$.

Indeed, $\bar \partial (\alpha_{1,0} - \omega_{1,0})=0$, so the coefficients
of $\gamma_{1,0}$ are holomorphic functions, and 
$$
\sum_{j=1}^N g_j f_j =
d_f (\gamma_{1,0}) = d_f \left(\omega _{1,0}- d_f \alpha_{2,0}\right)
= d_f (\omega _{1,0}) = \omega _{0,0}=1.
$$
The bound on the $g_j$ follows from the bounds on $\|\alpha_{1,0}\|$
and $\|\omega_{1,0}\|$ and Lemma \ref{partition},
which gives $C_f \le C \frac{N^2}\delta$, $D'_f \le C \frac{N^4}{\delta^2}$.
\end{proof*}

\end{document}